\documentclass{amsart}
\usepackage{amssymb,latexsym, amsmath, amscd, array, graphicx}

\swapnumbers
\numberwithin{equation}{section}

\theoremstyle{plain}

\newtheorem{thm}{Theorem}[section]
\newtheorem{cor}[thm]{Corollary}

\newtheorem{lem}[thm]{Lemma}
\newtheorem{lemma}[thm]{Lemma}
\newtheorem{conjec}[thm]{Conjecture}
\newtheorem{question}[thm]{Problem}
\newtheorem{prop}[thm]{Proposition}

\theoremstyle{definition}

\newtheorem{rem}[thm]{Remark}
\newtheorem{remark[thm]}{Remark}

\def\im{\protect\operatorname{im}}

\def\dim{\protect\operatorname{dim}}

\def\Int{\protect\operatorname{Int}}
\def\ker{\protect\operatorname{ker}}

\def\cat{\protect\operatorname{cat}}

\def\C{{\mathbb C}}
\def\Z{{\mathbb Z}}

\def\N{{\mathbb N}}

\def\1{\hbox{\rm\rlap {1}\hskip.03in{\rom I}}}
\def\Bbbone{{\rm1\mathchoice{\kern-0.25em}{\kern-0.25em}
{\kern-0.2em}{\kern-0.2em}I}}

\long\def\forget#1\forgotten{} %

\begin{document}

\title[The LS-category of the product of lens spaces]
{The LS-category of the product of lens spaces}
\author[A.~Dranishnikov]
{Alexander Dranishnikov}
\address{A. Dranishnikov, Department of Mathematics, University
of Florida, 358 Little Hall, Gainesville, FL 32611-8105, USA}
\email{dranish@math.ufl.edu}
\thanks{Supported by NSF grant DMS-1304627. The author would like to thank the Max-Planck Institut f\"{u}r Mathematik for the hospitality}

\begin{abstract}
We reduce Rudyak's conjecture that a degree one map between closed manifolds cannot raise the Lusternik-Schnirelmann category
to the computation of the category of the product of two lens spaces $L^n_p\times L_q^n$ with relatively prime $p$ and $q$. We have computed $\cat(L^n_p\times L^n_q)$ for values of $p,q>n/2$. It turns out that our computation supports the conjecture. 

For spin manifolds $M$ we establish a criterion for the equality $\cat M=\dim M-1$ which is a K-theoretic refinement of the Katz-Rudyak criterion
for $\cat M=\dim M$. We apply it to obtain the inequality $\cat(L^n_p\times L^n_q)\le 2n-2$ for all odd $n$ and odd relatively prime $p$ and $q$.
\end{abstract}

\maketitle

\section{Introduction}

This paper was motivated by the following conjecture from Rudyak's paper~\cite{R1}
\begin{conjec}
A degree one map between closed manifolds  cannot raise the Lusternik-Schnirelmann category.
\end{conjec}
It is known that  degree one maps $f:M\to N$ between manifolds tend to have the domain more complex than the image. The Lusternik-Schnirelmann category is a numerical invariant that measures the complexity of a space. Thus,  Rudyak's
conjecture that $\cat M\ge \cat N$ for a degree one map $f:M\to N$ is quite  natural. 
In~\cite{R1} (see also the book~\cite{CLOT}, page 65) Rudyak obtained some partial results supporting the conjecture. 
In particular, he proved the following
\begin{thm}[\cite{R1}]\label{rud} Let $f:M\to N$ be a degree $\pm 1$ map between closed  stably parallelizable
$n$-manifolds, $n\ge 4$, such that  $2\cat N\ge n+4$. Then $\cat M\ge \cat N$.
\end{thm}

In this paper we reduce Rudyak's conjecture to the following question about the LS-category of the product of two
$n$-dimensional  lens spaces ($n=2k-1$).

\begin{question}\label{qu}
Do there exist $n$ and relatively prime $p$ and $q$ such that $$\cat(L^{n}_p\times L^{n}_{q})> n+1\ ?$$
\end{question}
We show that an affirmative answer to this problem  gives a counter-example to Rudyak's conjecture.

This paper is devoted to computation of the category of the product $L^{n}_p\times L^{n}_{q}$ of lens spaces
for relatively prime $p$ and $q$. Here we use an abbreviated notation $L^n_p=L^n_p(\ell_1,\dots,\ell_k)$ for a general  lens space of dimension $n=2k-1$ defined for the linear $\Z_p$-action on $S^n\subset\C ^k$ determined by the set of natural numbers $(\ell_1,\dots,\ell_k)$ with $(p,\ell_i)=1$ for all $i$.

The obvious inequality $\cat X\le \dim X$ and the cup-length lower bound~\cite{CLOT} give  the 
estimates
 $$(*)\ \ \ \ \ \ \ \ \ \ \ \ \ \ n+1\le\cat(L^{n}_p\times L^{n}_{q})\le 2n.$$ 

In this paper we prove that for fixed $n$ almost always the lower bound is sharp.
\begin{thm}\label{m}
 For every $n=2k-1$ and primes $p,q\ge k$, $p\ne q$, for all lens spaces $L^n_p$ and $L^n_q$, 
$$\cat(L^{n}_p\times L^{n}_{q})=n+1.$$
\end{thm}

This result still leaves some hope to have $\cat(L^n_p\times L^n_q)>n+1$  for small values of $p$ (especially for $p=2$) for some lens spaces. 

In the second part of the paper we make an improvement of the upper bound in $(*)$.  The first improvement
comes easy by virtue of the Katz-Rudyak criterion~\cite{KR}: {\em For a closed $m$-manifold $M$ the inequality $\cat(M)\le m-1$ 
holds if and only if $M$ is inessential.}  We recall that Gromov calls a $m$-manifold $M$ {\em inessential} if a map $u:M\to B\pi$ that classifies its universal covering can be deformed to the $(m-1)$-dimensional skeleton $B\pi^{(m-1)}$. Since for relatively prime $p$ and $q$ the product $L^n_p\times L^n_q$ is inessential,
we have $\cat(L^n_p\times L^n_q)\le 2n-1$. In the paper we improve this inequality to the following:
\begin{thm}
For all odd $n$ and odd relatively prime $p$ and $q$,
$$
\cat(L^n_p\times L^n_q)\le 2n-2.$$
\end{thm}
For that we study a general question: When the LS-category of a closed spin $m$-manifold $M$ is less than $m-1$?  We proved that for a closed $m$-manifolds $M$ with $\pi_2(M)=0$ the inequality $\cat M\le m-2$ holds if and only if the map $u:M\to B\pi$ can be deformed to the $(m-2)$-dimensional skeleton $B\pi^{(m-2)}$. A deformation of a classifying map of a manifold to the $(m-2)$-skeleton $B\pi^{(m-2)}$ is closely related to
Gromov's conjecture on manifolds with positive scalar curvature and it was investigated in~\cite{BD}. Combining this with some ideas from~\cite{BD} we produce a criterion for
when a closed spin $m$-manifold $M$ has $\cat M\le m-2$. The criterion 
involves the vanishing of the integral homology and  $ko$-homology fundamental classes of $M$ under a map classifying the universal covering of $M$.
\begin{thm}[Criterion]
 For a closed spin inessential $m$-manifold $M$ with $\pi_2(M)=0$,  
$$\cat M\le\dim M-2$$ if and only if  $j_*u_*([M]_{ko})= 0$ where $j:B\pi\to B\pi/B\pi^{(m-2)}$ is the quotient map.
\end{thm}

Since a closed orientable manifold $M$ is inessential
if and only if $u_*([M])=0$ in $H_*(B\pi)$~\cite{Ba}, the Katz-Rudyak criterion for orientable manifolds can be rephrased as follows:
$\cat M\le m-1$ if and only if $u_*([M])=0$. Thus, our criterion is a further refinement of the  Katz-Rudyak criterion. 

It turns out that the vanishing of $u_*([M])$ in $H_*(B\pi)$ makes the primary obstruction to a deformation of 
$u:M\to B\pi$ to $B\pi^{(m-2)}$ trivial. It is not difficult to show that the second obstruction lives in the group of coinvariants $\pi_m(B\pi,B\pi)_{\pi}$
(see~\cite{BD}).
We prove that  the group of coinvariants $\pi_m(B\pi,B\pi^{(m-2)})_{\pi}$ naturally injects into the homotopy group $\pi_m(B\pi/B\pi^{(m-2)})$.
This closes a gap in the computation of the second obstruction in~\cite{BD}.
Based on that injectivity result we use the real connective K-theory to express the second obstruction in terms of the image of the $ko$-fundamental class. The spin condition is needed for the existence of a fundamental class in {\em ko}-theory.

The new upper bound implies that $\cat(L^3_p\times L^3_q)=4$ for all $p$ and $q$. Note that for prime $p$ and $q$ this fact can be also 
derived from Theorem~\ref{m}.

We complete the paper with a proof of the upper bound formula for the category of a connected sum of two manifolds:
\begin{thm}
$$
\cat{M\# N}\le\max\{\cat M,\cat N\}.
$$
\end{thm}
Since we use this formula in the paper and its original proof in~\cite{N} does not cover all  cases, we supply an alternative proof. 

\section{Preliminaries}

\subsection{LS-category.} The Lusternik-Schnirelmann category $\cat X\le k$ for a topological space $X$ if there is a cover $X=U_0\cup\dots\cup U_k$ by $k+1$ open subsets each of which is contractible in $X$. The subsets contractible in $X$ will be called in this note {\em $X$-contractible} and the covers of $X$ by $X$-contractible sets will be called {\em categorical}.

Let $\pi=\pi_1(X)$. We recall that the cup product $\alpha\smile\beta$ of twisted cohomology classes $\alpha\in H^i(X;L)$ and $\beta\in H^j(X;M)$ takes value in $H^{i+j}(X;L\otimes M)$ where $L$ and $M$ are $\pi$-modules and $L\otimes M$ is the tensor product over $\Z$~\cite{Bro}. Then the cup-length of $X$, denoted as $c.l.(X)$, is defined as the maximal integer $k$ such that $\alpha_1\smile\dots\smile\alpha_k\ne 0$ for some $\alpha_i\in H^{n_i}(X;L_i)$ with $n_i>0$.
The following inequalities give  estimates on the LS-category~\cite{CLOT}:
\begin{thm}\label{cl}
$c.l.(X)\le\cat X\le\dim X.$
\end{thm}

\subsection{Ganea-Schwarz approach to the LS-category}
Given two maps~$f_1:X_1\to Y$ and~$f_2:X_2\to Y$, we set
\[
Z=\{(x_1,x_2,t)\in X_1\ast X_2\mid f_1(x_1)=f_2(x_2)\}
\]
and define the {\em fiberwise join}, or {\em join over~$Y$} of~$f_1$, $f_2$
as the map
\[
f_1{\ast_Y}f_2:Z\to Y,\quad (f_1{\ast_X}f_2)(x_1,x_2, t)=f_1(x_1)=f_2(x_2).
\]

Let~$p_0^X:PX\to X$ be the Serre path fibration. This means that
$PX$ is the space of paths on~$X$ that start at the base point~$x_0\in X$, and~$p_0(\alpha)=\alpha(1)$.  We denote by
$p_n^X:G_n(X)\to X$ the iterated fiberwise join of $n+1$ copies of $p_0^X$.
Thus, the fiber $F_n=(p^X_n)^{-1}(x_0)$ of the fibration $p^X_n$ is the join product $\Omega X\ast\dots\ast\Omega X$ of $n+1$ copies of the loop space
$\Omega X$ on $X$. So, $F_n$ is $(n-1)$-connected.
It is known that $G_n(X)$ is homotopy equivalent to the mapping cone of the inclusion of the fiber $F_{n-1}\to G_{n-1}(X)$.

When $X=K(\pi,1)$, the loop space $\Omega X$ is naturally homotopy equivalent to $\pi$ and the space $G_n(\pi)=G_n(K(\pi,1))$ has the homotopy type of
a $n$-dimensional complex.

The proof of the following theorem can be found in \cite{CLOT}.

\begin{thm}[Ganea, Schwarz]\label{ganea}
\label{t:ganea}
For a~CW-space~$X$,~$\cat(X)\le n$ if and only if there exists a
section of~$p_n^X:G_n(X)\to X$.
\end{thm}
This theorem can be extended to maps:
\begin{thm}\label{ganea-maps}
\label{t:ganea}
For a map $f:Y\to X$ to a CW-space~$X$,~$\cat(f)\le n$ if and only if there exists a
lift of $f$ with respect to $p_n^X:G_n(X)\to X$.
\end{thm}
We recall that the LS-category of a map $f:Y\to X$ is a minimal number $k$ such that $Y$ can be covered by
$k+1$ open sets $U_0,\dots, U_k$ such that the restrictions $f|_{U_i}$ are null-homotopic for all $i$.

\

We use the notation
$\pi_*(f)=\pi_*(M_f,X)$ where $M_f$ is the mapping cylinder of~$f:X\to Y$.  Then~$\pi_i(f)=0$ for~$i\le n$ amounts to
say that it induces isomorphisms~$f_*:\pi_i(X)\to \pi_i(Y)$ for
$i\le n$ and an epimorphism in dimension~$n+1$. 
The following was proven in~\cite{DKR}.

\begin{prop}\label{join}
Let~$f_j:X_j\to Y_j$,~$3\le j\le s$ be a family of maps of~CW-spaces such
that~$\pi_i(f_j)=0$ for~$i\le n_j$.  Then the joins
satisfy
\[
\pi_k(f_1\ast f_2\ast\dots\ast f_s)=0
\]
for~$k\le\min\{n_j\}+s-1$.
\end{prop}

\subsection{The Berstein-Schwarz class.} Let $\pi$ be a discrete group and  $A$ be a $\pi$-module. By $H^*(\pi,A)$ we denote the cohomology of the group $\pi$ with coefficients in $A$ and by
$H^*(X;A)$ we denote the cohomology of a space $X$ with the twisted coefficients defined by $A$.
The Berstein-Schwarz class of a group $\pi$ is a certain cohomology class $\beta_{\pi}\in H^1(\pi,I(\pi))$ where $I(\pi)$ is the augmentation ideal of the group ring $\Z\pi$~\cite{Be},\cite{Sch}. 
It is defined as the first obstruction to a lift of $B\pi=K(\pi,1)$ to the universal covering $E\pi$. For a fixed set $S$
of generators of $\pi$, the 1-skeleton of $B\pi$ is the wedge of circles labeled by $S$  with $E\pi^{(1)}$ equals the Cayley graph $C$. The vertices of $C$ can be identified with $\pi$. Then $\beta_{\pi}$ is defined by a cocycle $\beta$ that takes every edge $[a,b]\subset C$ to $b-a\in I(\pi)$. 

Here is a more algebraic definition of $\beta_{\pi}$. Consider the  cohomology long exact sequence 
generated by the short exact sequence of coefficients
$$
0\to I(\pi)\to\Z\pi\stackrel{\epsilon}\to\Z\to 0
$$ where $\epsilon$ is the augmentation homomorphism.
Then $\beta_{\pi}=\delta(1)$ equals the image of the generator $1\in H^0(\pi;\Z)=\Z$ under the connecting homomorphism 
$\delta:H^0(\pi;\Z)\to H^1(\pi;I(\pi))$.

\begin{thm}[Universality~\cite{DR},\cite{Sch}]
For any cohomology class $\alpha\in H^k(\pi,L)$ there is a homomorphism of $\pi$-modules $I(\pi)^k\to L$ such that the induced homomorphism for cohomology takes $(\beta_{\pi})^k$ to $\alpha$  where $I(\pi)^k=I(\pi)\otimes\dots\otimes I(\pi)$ and $(\beta_{\pi})^k=\beta_{\pi}\smile\dots\smile\beta_{\pi}$.
\end{thm}

\begin{cor}\label{cd} 
For any group $\pi$ its cohomological dimension can be expressed as follows:
$$
cd(\pi)=\max\{n\mid (\beta_{\pi})^n\ne 0\}.
$$
\end{cor}

\begin{cor}[\cite{Sch}]\label{ganea-obstr}
The class  $(\beta_{\pi})^{n+1}$ is the primary obstruction to a section of $p^{B\pi}_n:G_n(\pi)\to B\pi$.
\end{cor}

\section{Some examples of degree one maps}

Let $M$ be an oriented manifold and $k\in\Z\setminus\{0\}$; by $kM$ we denote the connected sum $M\#\cdots\# M$ of $|k|$ copies of
$M$ taken with the opposite orientation if $k$ is negative. For an odd $n$ and natural $p>1$ we denote by $L^{n}_p$ a lens space, i.e. the orbit space
$S^{n}/\Z_p$ for a free linear action of $\Z_p=\Z/p\Z$ on the sphere $S^{n}$.

\begin{thm}\label{ex}
For $m,n\in 2\N+1$ and any  relatively prime numbers $p$ and $q$ there are $k,l\in\Z$ such that the manifold
$$
M=k(L^{m}_p\times S^{n})\# l(S^{m}\times L^{n}_q)
$$
admits a degree one map $\phi:M\to N$ onto $N=L^{m}_p\times L^{n}_q$.
\end{thm}
\begin{proof}
Take $k$ and $l$ such that $lp+kq=1$. Let $f:S^{m}\to L^{m}_p$ and $g:S^{n}\to L^{n}_q$
be the projections to the orbit space for $\Z_p$ and $\Z_q$  free actions respectively.
We may assume that the above connected sum is obtained by taking the wedge of $(|k|+|l|-1)$-spheres of dimension $(2n-1)$ embedded in one of the summands and gluing all other summands along those spheres.
Consider the quotient map
$$
\psi:k(L^{m}_p\times S^{n})\# l(S^{m}\times L^{n}_q)\to \bigvee_k(L^{m}_p\times S^{n})\vee \bigvee_l(S^{m}\times L^{n}_q)
$$
that collapses the wedge of those $(2n-1)$-spheres  to a point.
Let the map $$\phi:\bigvee_k(L^{m}_p\times S^{n})\vee \bigvee_l(S^{m}\times L^{n}_q)\to L^{m}_p\times L^{n}_q$$
be defined as the union 
$$
\phi=\bigcup_l(1\times g)\cup\bigcup_k(f\times 1).
$$
Note that the degree of $f\times 1$ is $p$, the degree of $1\times g$ is $q$, and the degree of
$\phi\circ\psi$ is $lp+kq=1$.
\end{proof}

\begin{prop} For $m\le n$,
$\cat(k(L^{m}_p\times S^{n})\# l(S^{m}\times L^{n}_q))=n+1.$
\end{prop}
\begin{proof}
It follows from the cup-length estimate that $\cat(S^m\times L^n_q)\ge n+1$ and generally,
$\cat(k(L^{m}_p\times S^{n})\# l(S^{m}\times L^{n}_q))\ge n+1$ when $l\ne 0$.
By the product formula, $\cat(S^m\times L^n_r)\le n+1$. Thus, $\cat(S^m\times L^n_r)= n+1$.
Then by the sum formula~\cite{N} (see Theorem~\ref{th1}), $$\cat(k(L^{m}_p\times S^{n})\# l(S^{m}\times L^{n}_q))\le n+1.$$
\end{proof}

Now one can see the connection between Rudyak's conjecture and Problem~\ref{qu}. If there exist relatively prime $p$ and $q$ and odd $n$ such that $\cat(L^n_p\times L^n_q)>n+1$, then the map of Theorem~\ref{ex} will be a counter-example to Rudyak's conjecture.

\begin{rem}
In Theorem~\ref{ex} one can use fake lens spaces. Since every fake lens space is homotopy equivalent to a lens space~\cite{Wa} and the LS-category is a homotopy invariant, it suffices to consider only the classical lens spaces.
\end{rem}

\section{On the category of the product of lens spaces}
Let $\bar\ell=(\ell_1,\dots,\ell_k)$ be a set of  mod $p$ integers relatively prime to $p$. The lens space $L^{2k-1}_p(\bar\ell)$ is the orbit space
of the action of $\Z_p=\langle t\rangle$ on the unit sphere $S^{2k-1}\subset\C^k$ defined by the formula:
$$
t(z_1,\dots, z_k)=(e^{2\pi i\ell_1/p}z_1,\dots, e^{2pi i\ell_k/p}).$$
We note that for all $k$ the lens spaces $L^{2k-1}_p(\bar\ell)$ have a natural CW-complex structure with one cell in each dimension $\le 2k-1$ such that
$L^{2k-1}_p(\bar\ell)$ is the $(2k-1)$-skeleton of $L^{2k+1}_p(\bar\ell,\ell_{k+1})$.
If $\alpha:\Z_p\times S^{2k-1}\to S^{2k-1}$ is a free action which is not necessarily linear, its orbit space is called a {\em fake lens space}
and is denoted by $L^{2k-1}_p(\alpha)$.

We recall that a closed oriented $n$-manifold $M$ is called {\em inessential}~\cite{Gr}
if a map $u:M\to B\pi=K(\pi,1)$ that classifies its universal cover can be deformed to the $(n-1)$-dimensional skeleton $B\pi^{(n-1)}$.
It is known that a closed oriented $n$-manifold $M$ is essential if and only if $u_*([M])\ne 0$ where $[M]\in H_n(M;\Z)$ denotes the fundamental class~\cite{Ba},\cite{BD}. 

We note that $\cat M=\dim M$ if and only if $M$ is essential~\cite{KR}.
Clearly, every lens space $L^{n}_p$ is essential. In particular, $\cat L^{n}_p=n$. Since $\Z_p\otimes\Z_q=0$ for relatively prime $p$ and $q$,
the product $ L^{m}_p\times L^{n}_q$ is inessential. Hence, $\cat(L^{m}_p\times L^{n}_q)\le m+n-1$
for all $p$ and $q$.

\subsection{Stably parallelizable lens spaces} First we make our computation for stably parallelizable lens spaces.
\begin{prop}\label{parallel}
For  lens spaces $L^m_p$ and $L^n_q$, $m\le n$, with $(p,q)=1$ which are homotopy equivalent to stably parallelizable manifolds,
$$\cat(L^m_p\times L^n_q)=n+1.$$
\end{prop}
\begin{proof} Let $$\phi:M=k(L^{m}_p\times S^{n})\# l(S^{m}\times L^{n}_q)\to N=L^m_p\times L_q^n$$ be the map of degree one from Theorem~\ref{ex}.
Suppose that $L^m_p$ and $L^n_q$ are homotopy equivalent to stably parallelizable manifolds $N_p^m$ and $N_q^n$ respectively.
Then there are homotopy equivalences $h:M'=k(N^{m}_p\times S^{n})\# l(S^{m}\times N^{n}_q)\to M$
and $h':N=L^n_q\to N'=N_p^m\times N_q^n$. Since a connected sum and the product of stably parallelizable manifolds
are stably parallelizable  (see for example~\cite{Ko}), the manifolds $M'$ and $N'$ are stably parallelizable.
Assume that $\cat(L^m_p\times L^n_q)\ge n+2$. Then $$2\cat N'=2\cat(L^m_p\times L^n_q)=2(n+2)\ge m+n+4=\dim(L^m_p\times L^n_q)+4.$$
By Theorem~\ref{rud} applied to the map $h'\circ\phi\circ h:M'\to N'$ from Theorem~\ref{ex} we obtain  a contradiction: $$n+2=\cat N=\cat N'\le\cat M'=\cat M=n+1.$$
\end{proof}
Since all orientable 3-manifolds are stably parallelizable, we obtain
\begin{cor}\label{3} For relatively prime $p$ and $q$.
$$\cat(L^3_p\times L^3_q)=4.$$
\end{cor}

There is a characterization of stable parallelizability of lens spaces~\cite{EMSS}:
The lens space $L^{2k-1}_p(\ell_1,\dots,\ell_k)$ is stably parallelizable if and only if $p\ge k$ and
$b_1^{2j}+\dots+b_k^{2j}=0$ mod $p$ for $j=1,2,\dots,[\frac{k-1}{2}]$.
We recall that two lens spaces $L^{2k-1}_p(\ell_1,\dots,\ell_k)$ and
$L^{2k-1}_p(\ell_1,\dots,\ell_k)$ are homotopy equivalent~\cite{Ol} if and only if the mod $p$ equation
$$\ell_1\ell_2\dots\ell_k=\pm a^k\ell_1'\ell_2'\dots\ell_k'$$ has a  solution $a\in\Z_p$. 
These conditions imply that  a lens space is  rarely homotopy equivalent to a stably parallelizable one.
Nevertheless, it was shown in~\cite{EMSS} that each $n=2k-1$ for infinitely many primes $p$
there are stably parallelizable lens spaces $L^n_p$. Clearly, there are more chances for the existence of stably parallelizable fake lens spaces with given $n$ and $p$. 
Thus, in~\cite{Kw} Kwak proved that for every odd $n=2k-1$ and $p\ge k$ there is a fake $n$-dimensional stably parallelizable lens space. Since every fake lens space is homotopy equivalent to a lens space~\cite{Wa}, we obtain that for every $n=2k-1$ and $p\ge k$ there is a lens space $L^n_p$ homotopy equivalent to a stably parallelizable manifold.

\subsection{Category of classifying maps}
We recall that any map $u:X\to B\pi=K(\pi,1)$ of a CW complex $X$ that induces an isomorphism of the fundamental group classifies the universal covering
$\tilde X$, i.e., $\tilde X$ is obtained as the pull-back of the universal covering $E\pi$ of $B\pi$ by means of $u$.
We call such a map {\em  a classifying map} of $X$.
\begin{prop}\label{deform}
Let $u:X\to B\pi$ be a map classifying the universal covering of a CW complex $X$. Then the following are equivalent:

(1) $\cat(u)\le k$; 

(2) $u$ admits a lift $u':X\to G_k(\pi)$ of $u$ with respect to $p^{\pi}_n:G_k(\pi)\to B\pi$;

(3) $u$ is homotopic to a map $f:X\to B\pi$ with $f(X)\subset B^{(k)}$.
\end{prop}
\begin{proof}
(1) $\Rightarrow$ (2) is a part of Theorem~\ref{ganea-maps}.

(2) $\Rightarrow$ (3)
Since $G_k(\pi)$ has the homotopy type of a $k$-dimensional complex, the map $p^{\pi}_k$ can be deformed to a map $p'$ with the image in $B\pi^{(k)}$.
Then we can take $f=p'\circ u'$.

(3) $\Rightarrow$ (1) For a map $f:X\to B\pi$ with $f(X)\subset B^{(k)}$ homotopic to $u$ we obtain $\cat(u)=\cat(f)\le\cat B\pi^{(k)}\le k$.
\end{proof}

\begin{thm}\label{space-map} 
Let $X$ be an $n$-dimensional CW complex with a classifying map $u:X\to B\pi$ having $\cat u=k$ and with $m$-connected  universal covering
$\tilde X$. Suppose that $\cat(u)+m\le n$. Then
$$
\cat X=\cat(u).
$$
\end{thm}
\begin{proof}
Clearly, $\cat X\ge\cat u$.
We show that $\cat X\le k=n-m$.

Note that the map $p^X_{k}$ factors through the pull-back,
$p^X_{k}=p'\circ q$:
$$
\begin{CD}
G_{k}(X) @>q>> Z @>f'>> G_{k}(B\pi)\\
@. @Vp'VV @Vp^{\pi}_{n-2}VV\\
@. X @>u>> B\pi.\\
\end{CD}
$$
The condition $\cat u=k$ implies that $u$ has a lift $u'\to G_{k}(B\pi)$, $u=p^{\pi}_{k}u'$.
Hence, $p'$ admits a section $s:X\to Z$.
Since $X$ is $n$-dimensional, to show that $s$ has a lift with respect to $q$ it suffices to prove that the homotopy fiber $F$ of the map $q$ is $(n-1)$-connected.  Note that the homotopy exact sequence of the fibration
$$
F\to (p^X_{k})^{-1}(x_0)\stackrel{\bar f}\to (p^{\pi}_{k})^{-1}(y_0)
$$
where $\bar f$ is the restriction of $f'\circ q$ to the fiber $(p^X_{k})^{-1}(x_0)$ coincides with that of
$$
F'\to\ast_{k+1}\Omega(M)\stackrel{\ast\Omega(u)}\longrightarrow \ast_{k+1}\Omega(B\pi).
$$
Since $\pi_i(X)=0$ for $i\le n-k$, we obtain $\pi_i(\Omega u)=0$ for $i\le n-k-1$. 
Then by Proposition~\ref{join}, $\pi_{i}(\ast_{k+1}\Omega u)=0$ for $i\le k+(n-k-1)=n-1$ and hence $\pi_{i}(F)=0$
for $i\le n-1$.
Thus, $s$ admits a lift and, therefore, there is a section of $p^X_{k}$. Therefore, by Theorem~\ref{ganea} $\cat X\le k$.
\end{proof}

\subsection{The main computation}
\begin{prop}\label{map}
For any two lens spaces $L^n_p(\bar\ell)$ and $L^n_p(\bar\mu)$ there is a map $f:L^n_p(\bar\ell)\to L^n_p(\bar\mu)$
that induces isomorphism of the fundamental groups.
\end{prop}
\begin{proof}
Let $q_1:S^n\to L^n_p(\bar\ell)$ and $q_2:S^n\to L^n_p(\bar\mu)$ be the projection onto the orbit spaces of the corresponding $\Z_p$-actions.
We note that $L^n_p(\bar\mu)$ is the $n$-skeleton in $L^{n+2}_p(\bar\mu,1)$. Let $\bar q_2:S^n+2\to L^n_p(\mu,1)$ be the corresponding projection.
Since in the Borel construction for $\Z_p$ action on $S^n$ and $S^{n+2}$ the projection $p_1$ is $(n+1)$-connected
$$
\begin{CD}
S^n @<<< S^n\times S^{n+2} @>>> S^{n+2}\\
@ Vq_1VV @ VqVV @ V\bar q_2VV\\
L^n_p(\bar\ell) @<p_1<< S^n\times_{\Z_p} S^{n+2} @>p_2>>L^{n+2}(\bar\mu, 1)\\
\end{CD}
$$
it admits a section $s:L^n_p(\bar\ell)\to S^n\times_{\Z_p} S^{n+2}$. Then $f$ is a cellular approximation of $p_2\circ s$.
\end{proof}

\begin{thm}\label{main}
For every odd $n=2k-1$ and distinct primes $p,q\ge k$,
$$\cat(L^{2k-1}_p\times L^{2k-1}_q)=n+1.$$ 
\end{thm}
\begin{proof} Let $L^n_p=L^n_p(\bar\ell)$ and $L^n_q(\bar\ell')$ for $\bar\ell=(\ell_1,\dots,\ell_k)$ and
$\bar\ell'=(\ell'_1,\dots,\ell'_k)$.
By Kwak's theorem (Theorem 3.1~\cite{Kw}) there are stably parallelizable fake lens spaces $L_p^n(\alpha)$ and $L^n_q(\alpha')$.
By Wall's theorem they are homotopy equivalent to lens spaces $L_p^n(\bar\mu)$ and $L^n_q(\bar\mu')$
for some $\bar\mu$ and $\bar\mu'$. By Proposition~\ref{parallel}, $\cat(L_p^n(\bar\mu)\times L_q^n(\bar\mu'))=n+1$.
By Proposition~\ref{deform}, there is a classifying map $u:L_p^n(\bar\mu)\times L_q^n(\bar\mu')\to B\Z_{pq}^{(n+1)}$.
By Proposition~\ref{map} there are maps $f_p:L^n_p\to L^n_p(\bar\mu)$ and $f_q:L^n_q\to L^n_q(\bar\mu')$ that induce an
isomorphism of the fundamental groups. Therefore, $$u'=u\circ(f_p\times f_q):L^n_p\times L^n_q\to B\Z_{pq}^{(n+1)}$$
is a classifying map for $L^n_p\times L^n_q$. Hence, $\cat(u')\le n+1$. Since the universal covering of the space $L^n_p\times L^n_q$ is $(n-1)$-connected, by Theorem~\ref{space-map}
we obtain $\cat(L^n_p\times L^n_q)=n$.
\end{proof}
\begin{rem}
When $p$ and $q$ are relatively prime but not necessarily prime we can prove the equality $\cat(L^{n}_p\times L^{n}_q)=n+1$ with a stronger restriction $p,q\ge n+3$. We do not present the proof, since it is more technical. It consists of computation of obstructions for deforming a classifying map $u:L^{n}_p\times L^{n}_q\to B\Z_{pq}$ to the $(n+1)$-skeleton. Vanishing of the first obstruction happens without any restriction on $p$ and $q$. Since
it is a curious fact on its own it is presented in the next section. The higher obstructions vanish due to 
the fact that cohomology groups of $\Z_{pq}$ are $pq$-torsions and a theorem of Serre~\cite{Se} which  states that the group $\pi_{n+k}(S^n)$ has zero $r$-torsion component for $k < 2r-4$.
\end{rem}
We note that Theorem~\ref{main} can be stated  for all lens spaces $L^n_p$ with values of $n$ and $p$ for which
there exists a stably parallelizable fake lens space $L^n_p(\alpha)$. 
\begin{question}
For which values of $n$ and and $p$ is there  a stably parallelizable fake lens space $L^n_k(\alpha)$ ?
\end{question}
This does not seem to happen very often when $p=2$. At least a real $(2k-1)$-dimensionl projective space is stably parallelizable if and only if
$k=1,2$, or $4$.

\section{The Berstein-Schwarz class for the product of finite cyclic groups} 

Let $u:L^n_p\times L^n_q\to B\Z_{pq}$ be a classifying map. By Theorem~\ref{space-map} and the fact that $\cat(L^n_p\times L^n_q)\ge n+1$,
the condition $\cat(u)\le n+1$ is equivalent to
the equality $\cat(L^n_p\times L^n_q)=n+1$. By Proposition~\ref{deform} the inequality $\cat(u)\le n+1$ is equivalent to the existing of a lift $u'$ of $u$ with respect to $p_n:G_{n+1}(\Z_{pq})\to B\Z_{pq}$. In view of Corollary~\ref{ganea-obstr} the primary obstruction to such lift is $u^*(\beta^{n+2})$ where
$\beta$ is the Berstein-Schwarz class of $\Z_{pq}$. We prove that this obstruction is always zero and even more:
\begin{thm}\label{ber-nul}
For all $n$, and all relatively prime $p$ and $q$,
$$u^*(\beta^{n+1})=0.$$
\end{thm}
\begin{rem}One can show that for sufficiently large $p$ and $q$ the higher obstructions are trivial as well, since the homotopy groups of the fiber of $p_n^{\pi}$
do not contain $r$-torsions for large $r$. This would give a result similar to Theorem~\ref{main} which does not cover small values of $p$.
\end{rem}

We denote by $Z(m)$, $I(m)$, $\epsilon_m:\Z(m)\to\Z$, and $\beta_m$, the group ring $\Z Z_m$, the augmentation ideal, the augmentation, and
the Berstein-Schwarz class for $\Z_m$ respectively. Let $t_m=\sum_{g\in\Z_m}g\in\Z(m)$.
We use the same notation $t_m$ for a constant map $t_m:\Z_m\to\Z(m)$ with the value $t_m$.
We note that the group of invariants of $\Z(m)$ is $\Z$ generated by $t_m$. Thus, $H^{0}(\Z_m;\Z(m))=\Z$.

\begin{prop}\label{order p} 
Let $\beta_p$ denote the Berstein-Schwarz class for the  group $\Z_p=\Z/p\Z$.
Then $\beta_p$ has order $p$ and hence is $q$-divisible for any $q$ relatively prime to $p$.
\end{prop}
\begin{proof} Let $t\in\Z_p$ be a generator.
We note that $$H^0(\Z_p;\Z\Z_p)=(\Z\Z_p)^{\Z_p}=\Z\langle 1+t+\dots+t^{p-1}\rangle$$ is the group of invariants which
is isomorphic to the subgroup  of $\Z\Z_p$ generated by $1+t+\dots+t^{p-1}$. Then the augmentation homomorphism
$\epsilon:\Z\Z_p\to\Z$ induces a homomorphism $\epsilon_*:H^0(\Z_p;\Z\Z_p)\to H^0(\Z_p;\Z)=\Z$ that takes
the generator $1+t+\dots+t^{p-1}$ to $p$. Thus, $p\beta_p=p\delta(1)=\delta(p)=0$.
 
Note that $\beta_p$ generates a subgroup  $G$ of order $p$ in $H^1(\pi; I(\Z\pi))$.  Therefore it is $q$-divisible
for $q$ with $(p,q)=1$.
\end{proof}

We recall that the cross product  $$H^i(X;M)\times H^j(X';M')\to H^{i+j}(X\times X';M\otimes_{\Z} M')$$ is defined for any
$\pi_1(X)$-module $M$ and $\pi_1(X')$-module $M'$. Also we note that $$H^i(X;M\oplus M')=H^i(X;M)\oplus H^i(X;M').$$
\begin{prop}
For relatively prime $p$ and $q$ there are $k,l\in\Z$ such that the Berstein-Schwarz class $\beta_{pq}$ is the image
of the class $$(\beta_p\times l, k\times\beta_q)\in H^1(\Z_{pq};I(p)\otimes\Z(q))\oplus H^1(\Z_{pq};\Z(p)\otimes I(q))$$ under the coefficient homomorphism
$$\phi:I(p)\otimes\Z(q)\bigoplus\Z(p)\otimes I(q)\to I(pq)\subset\Z(pq)=\Z(p)\otimes\Z(q)$$ defined by the inclusions of the direct summands into $\Z(p)\otimes\Z(q)$ and the summation.
\end{prop}
\begin{proof} Let $k$ and $l$ be such that $kp+lq=1$. 

The addition in $\Z(pq)$ defines the following commutative diagram
$$
\begin{CD}
0 @>>> I(pq) @>>> \Z(pq) @>\epsilon_{pq}>>\Z @>>> 0\\
@. @A{\phi}AA @A+AA @A\epsilon_p+\epsilon_qAA @.\\
0 @>>> I(p)\otimes\Z(q)\oplus\Z(p)\otimes I(q) @>>>\Z(pq)\oplus\Z(pq) @>{\alpha}>>\Z(q)\oplus\Z_(p) @>>> 0
\end{CD}
$$
which defines a commutative square for cohomology
$$
\begin{CD}
H^0(\Z_{pq},\Z) @>\delta>> H^1(\Z_{pq},I(pq))\\
@A{\epsilon_*}AA @A{\phi_*}AA\\
H^0(\Z_{pq},\Z(p)\oplus\Z(q)) @>\delta'>> H^1(\Z_{pq},I(p)\otimes\Z(q)\oplus\Z(p)\otimes I(q)).\\
\end{CD}
$$
The homomorphism $\theta:\Z(pq)\to\Z(p)\oplus\Z(q))$ defined on the basis as $\theta(a\times b)=(lt_q,kt_p)$
is  a cochain since it is $\Z_{pq}$-equivariant. It is a cocycle, since it is constant. Note that
$(\epsilon_p+\epsilon_q)\circ\theta(a\times b)=kp+lq=1$ for any $a\in\Z_p$ and $b\in\Z_q$.
This means that the cohomology class $[\theta]$ is taken by $\epsilon_*$ to a generator $1\in H^0(\Z_{pq};\Z)$.
Then $\beta_{pq}=\delta(1)=\phi\delta([\theta])$.

Consider a $\Z(pq)$-homomorphism $\bar\theta:\Z(p)\times\Z(q)\to \Z(pq)\oplus\Z(pq)$ defined by the formula $\bar\theta(a\times b)=(a\times lt_q,kt_p\times b)$.
Since $\alpha(\bar\theta)=\theta$, by the Snake Lemma $\delta'([\theta])$ is defined by the 1-cocycle $\delta(\bar\theta):C_1\to  I(p)\otimes\Z(q)\oplus\Z(p)\otimes I(q)$. Note that the cellular 1-dimensional chain group $C_1$ is defined the Cayley graph $C$ of $\Z_{pq}$.
Thus, $C=C^p\times\Z_q\cup \Z_p\times C^q$ where $C^p$ and $C^q$ are the Cayley graphs (cycles) for $\Z_p$ and $\Z_q$ respectively.
Note that $$\delta(\bar\theta)([a_1,a_2]\times b)=\bar\theta((a_2-a_1)\times b=\bar\theta(a_2\times b-\bar\theta(a_1\times b)$$
$$=
(a_2\times lt_q, kt_p\times b)-(a_1\times lt_q, kt_p\times b)=((a_2-a_1)\times lt_q,0)=$$
$$(\beta_p\times lt_q)([a_1,a_2]\times b)=
(\beta_p\times lt_q, kt_p\times\beta_q)([a_1,a_2]\times b).$$ Similarly, we have the equality for edges of the  type $a\times[b_1,b_2]$.
Here $\beta_p$ and $\beta_q$ denote the canonical cochains that define the Berstein-Schwarz classes of $\Z_p$ and $\Z_q$.

Thus, $\delta'([\theta])=(\beta_p\times l,k\times\beta_q)$ in $$H^1(\Z_{pq}, I(p)\otimes\Z(q))\oplus H^1(\Z_{pq},\Z(p)\otimes I(q))= H^1(\Z_{pq}, I(p)\otimes\Z(q)\oplus\Z(p)\otimes I(q)).$$
\end{proof}

\subsection{Proof of Theorem~\ref{ber-nul}} We show that
$u^*(\beta_{pq}^{n+1})=0$ where $$u=i_p\times i_q:L^{n}_p\times L^{n}_q\to B\Z_p\times B\Z_q=B\Z_{pq}$$
is the inclusion.
Note that $(\beta_p\times lt_q,kt_p\times\beta_q)=\beta_p\times lt_q+kt_p\times\beta_q$.
Thus, it suffices to show that $u^*(\beta_p\times lt_q+kt_p\times\beta_q)^{n+1}=0$.
Note that  $$u^*(\beta_p\times l+k\times\beta_q)=i_p^*(\beta_p)\times l +k\times i^*_q(\beta_q).$$
Then $(i_p^*(\beta_p)\times l +k\times i^*_q(\beta_q))^{n+1}=(i_p^*(\beta_p)\times l)^{n+1} +(k\times i^*_q(\beta_q)^{n+1}+F$
where $F$ consists of monomials containing both factors.

Claim 1: $(i_p^*(\beta_p)\times l)^{n+1}=0$ and $(k\times i^*_q(\beta_q))^{n+1}=0$.

Proof: There is an automorphism of the coefficients 
$$
(I(p)\otimes\Z(q))\otimes\dots\otimes(I(p)\otimes\Z(q))\to I(p)\otimes\dots\otimes I(p)\otimes\Z(q)\otimes\dots\otimes\Z(q)
$$
that takes
 $(i_p^*(\beta_p)\times l)^{n+1}$ to $(i_p^*(\beta_p)^{n+1}\times l^{n+1}=0$.
Similarly, $(k\times i^*_q(\beta_q)^{n+1}=0$.

Claim 2: $(i_p^*(\beta_p)\times l)A(k\times i^*_q(\beta_q))=0$ for any $A$. 

Proof: Indeed, since $i_p^*(\beta_p)$ is divisible by $q$ (see Proposition~\ref{order p}),
$$
(i_p^*(\beta_p)\times l)A(k\times i^*_q(\beta_q))=(\frac{1}{q}(i_p^*(\beta_p)\times l))Aq(k\times i^*_q(\beta_q))=0.$$
Thus, $F=0$ and the result follows.
\qed

\section{On the category of  {\it ko}-inessential manifolds}

\subsection{ Deformation into the $(n-2)$-dimensional skeleton}
We recall that a classifying map $u:M\to B\pi$ of a closed orientable $n$-manifold $M$ can be deformed into the $(n-1)$-skeleton $B\pi^{(n-1)}$
if and only if $u_*([M])\ne 0$ where $[M]\in H_n(M;\Z)$ denotes an integral fundamental class~\cite{Ba}. 
In~\cite{BD} we proved the following proposition which sets the stage for computation of the second obstruction.
\begin{prop}\label{BD}
Every inessential $n$-manifold $M$ with a fixed CW structure on it admits a classifying map
$u:M\to B\pi$ with $u(M)\subset B\pi^{(n-1)}$ and $u(M^{(n-1)})\subset B\pi^{(n-2)}$.
\end{prop}

We postpone the proof of the following lemma to the end of the section.
\begin{lem}\label{coinvariants-mono}
For any group $\pi$ and a CW complex $B\pi$, for $n\ge 5$,  the  homomorphism induced by the quotient map
$$p_*:\pi_n(B\pi,B\pi^{(n-2)})\to \pi_n(B\pi/B\pi^{(n-2)})$$ factors through the group of coinvariants, $p_*=\bar p_*\circ q_*$,
$$\pi_n(B\pi,B\pi^{(n-2)})\stackrel{q_*}\to\pi_n(B\pi,B\pi^{(n-2)})_{\pi}\stackrel{\bar p_*}\to\pi_n(B\pi/B\pi^{(n-2)})$$ where $\bar p_*$ is injective.
\end{lem}
We recall that for a $\pi$-module $M$ the group of coinvariants is $M\otimes_{\Z\pi}\Z$.
\begin{rem}
In the proof of Lemma 4.1~\cite{BD} it was stated erroneously that $\bar p_*$ is bijective. This has not affect on the proof of the lemma, since
the injectivity of $bar p_*$ is sufficient for that to work.
\end{rem}

\begin{thm}\label{obstruction} 
Let $M$ be an $n$-manifold with a CW complex structure with one top dimensional cell. Suppose that
a classifying map $u:M\to B\pi$  satisfies the condition $u(M^{(n-1)})\subset B\pi^{(n-2)}$ and let $\bar u:M/M^{(n-1)}=S^n\to B\pi/B\pi^{(n-2)}$ be the induced map.
Then the following are equivalent:

(1) There is a deformation of $u$ in $B\pi$ to a map $f:M\to B\pi^{(n-2)}$;

(2) $\bar u_*(1)=0$ in $\pi_n(B\pi/B\pi^{(n-2)})$ where $1\in\Z=\pi_n(S^n)$.
\end{thm}
\begin{proof}
The primary obstruction to deforming $u$ to $B\pi^{(n-2)}$ is defined by the cocycle
$$c_u=u_*:\pi_n(M,M^{(n-1)})\to\pi_n(B\pi,B\pi^{(n-2)})$$ with the cohomology class
$o_u=[c_u]\in  H^n(M;\pi_n(B\pi,B\pi^{(n-2)}))$. By  Poincar\'{e} duality $o_u$ is dual
to the homology class $PD(o_u)\in H_0(M;\pi_n(B\pi,B\pi^{(n-2)}))=\pi_n(B\pi,B\pi^{(n-2)})_{\pi}$
represented by $q_*u_*(1)$ where $$q_*:\pi_n(B\pi,B\pi^{(n-2)})\to\pi_n(B\pi,B\pi^{(n-2)})_{\pi}$$ 
is the projection onto the group of coinvariants and $$u_*:\pi_n(M,M^{(n-1)})=\Z\to\pi_n(B\pi,B\pi^{(n-2)})$$
is induced by $u$. We note that $\pi_n(B\pi,B\pi^{(n-2)})=\pi_n(E\pi,E\pi^{(n-2)})$. By 
Theorem~\ref{coinvariants-mono} the homomorphism
$\bar p_*$ is injective. Hence $\bar p_*q_*u_*(1)=0$ if and only if $o_u=0$.
The commutative diagram
$$
\begin{CD}
\pi_n(M,M^{(n-1)}) @>u_*>>\pi_n(B\pi,B\pi^{(n-2)}) @>q_*>> \pi_n(B\pi,B\pi^{(n-2)})\\
@V=VV @. @V{\bar p_*}VV\\
\Z @>=>>\pi_n(M/M^{(n-1)}) @>\bar u_*>>  \pi_n(B\pi/B\pi^{(n-2)})\\
\end{CD}
$$
implies that
$\bar u_*(1)=\bar p_*q_*u_*(1).$
\end{proof}

\subsection{{\it ko}-inessential manifolds} 
We recall that an orientable closed $n$-manifold $M$ is inessential if and only if $u_*([M])\ne 0$ where $[M]\in H_n(M;\Z)$ is a fundamental class
and $u:M\to B\pi$ is a classifying map. We call a closed spin $n$-manifold $M$ {\em $ko$-inessential} if $u_*([M]_{ko})=0$ in $ko_n(B\pi)$
where $ko_*$ denotes the real connective K-theory homology groups.

We need the following proposition proven in~\cite{BD} where $\pi^s_*$ is the stable homotopy theory.

\begin{prop}\label{pi-ko-iso}
The natural transformation $\pi_*^s(pt)\to ko_*(pt)$ induces an
isomorphism $\pi_{n}^s(K/K^{(n-2)})\to ko_{n}(K/K^{(n-2)})$ for any CW
complex $K$.
\end{prop}

\begin{thm}\label{2-obstruction}
A classifying map $u:M\to B\pi$ of an inessential closed spin
$n$-manifold, $n>3$, is homotopic to a map $f:M\to
B\pi^{(n-2)}$ if and only if $j_*u_*([M]_{ko})=0$ in $ko_n(B\pi,B\pi^{(n-2)})$
where $[M]_{ko}$ is a $ko$-fundamental class.
\end{thm}
\begin{proof}
By Proposition~\ref{BD} a classifying map $u$ can be chosen to satisfy the condition $u(M^{(n-1)})\subset B\pi^{(n-2)}$.
We show that $\bar u_*(1)=0$ if and only if $j_*u_*([M]_{ko})=0$ and apply Theorem~\ref{obstruction}.

The restriction $n>3$ implies that $\bar
u_*(1)$ survives in the stable homotopy group. In view of
Proposition~\ref{pi-ko-iso}, the element $\bar u_*(1)$ survives in
the composition
$$\pi_n(B\pi/B\pi^{(n-2)}) \to \pi_n^s(B\pi/B\pi^{(n-2)})\to
ko_n(B\pi/B\pi^{(n-2)}).$$

The commutative diagram
$$
\begin{CD}
\pi_n(S^n) @>\bar u_*>> \pi_n(B\pi/B\pi^{(n-2)})\\
@V\cong VV @V\cong VV\\
\pi_n^s(S^n) @>\bar u_*>> \pi_n^s(B\pi/B\pi^{(n-2)})\\
@V\cong VV @V\cong VV\\
ko_n(S^n) @>\bar u_*>> ko_n(B\pi/B\pi^{(n-2)})\\
\end{CD}
$$
implies that $\bar u_*(1)= 0$ for $ko_n$ iff and only if $\bar u_*(1)=0$ for $\pi_n$. 

From the diagram with the quotient map $\psi:M\to M/M^{(n-1)}=S^n$
$$
\begin{CD}
ko_n(M) @>u_*>>ko_n(B\pi)\\
@V{\psi_*}VV @V{j_*}VV\\
ko_n(S^n) @>\bar u_*>>ko_n(B\pi/B\pi^{(n-1)})\\
\end{CD}
$$
it follows that $j_*u_*([M]_{ko})=\bar u_*\psi_*([M]_{ko})=\bar u_*(1)$.
Thus, $j_*u_*([M]_{ko}0=0$ if and only if $\bar u_*(1)=0$ for $n$-dimensional homotopy groups.
\end{proof}

We recall that spin manifolds are exactly those that admit an orientation in real connective K-theory 
$ko_*$. For spin manifolds we prove the following criterion.
\begin{thm}\label{cat=n-1} 
For a closed spin $n$-manifold $M$ with
$\cat M\le\dim M-2$,  $$j_*u_*([M]_{ko})= 0$$ in $ko_n(B\pi,B\pi^{(n-2)})$ where  
$u:M\to B\pi$  classifies the universal cover of $M$
and $j:(B\pi,\emptyset)\to (B\pi,B\pi^{(n-2)})$ is the inclusion.

For a closed spin inessential $n$-manifold $M$ with $\pi_2(M)=0$,  
$\cat M\le\dim M-2$ if and only if  $j_*u_*([M]_{ko})= 0.$
\end{thm} 
\begin{proof}
The inequality $\cat M\le n-2$ implies that
the map $u$ has a lift $u'\to G_{n-2}(B\pi)$, $u=p^{\pi}_{n-2}u'$. Since $G_{n-2}(B\pi)$ 
is homotopy equivalent to an $(n-2)$-dimensional complex, $p^{\pi}_{n-2}$ can be deformed to 
$p':G_{n-2}(B\pi)\to B\pi^{(n-2)}$. Thus $u$ can be deformed to $B\pi^{(n-2)}$. By
Theorem~\ref{2-obstruction}  $j_*u_*([M]_{ko})\ne 0$ . 

Now let $\pi_2(M)=0$, and $j_*u_*([M]_{ko})=0$.  
By Theorem~\ref{2-obstruction} the map $u$ can be deformed to a map
$f:M\to B\pi^{(n-2)}$. By Proposition~\ref{deform}, $\cat(u)\le n-2$. Since $\pi_2(M)=0$, the universal covering of $M$
is 2-connected. By Theorem~\ref{space-map}, $\cat M\le n-2$.
\end{proof}

\begin{prop}\label{vanish}
Let $M=L_p^{m}\times L_{q}^{n}$, $m,n>2$, be given a $ko$-orientation for some relatively prime $p$ and $q$ and let
$u:M\to B\Z_{pq}$ be  a classifying map of its universal cover. Then $u_*([M]_{ko})= 0$. 
\end{prop}
\begin{proof}
Note that $[M]_{ko}=(1+v)([L_p^m]_{ko}\times[L^n_q]_{ko})$ where $v\in \tilde {ko^*}(M)$ is from the reduced cohomology
$ko$-theory and the product is the cap product~\cite{R2}.
Therefore it suffices to show that $u^p_*([L_p^m]_{ko})\times u^q_*([L^n_q]_{ko})=0$ where $u^p:L^m_p\to B\Z_p$ and
$u^q:L^n_q\to B\Z_q$ are  classifying maps.
This equality follows from the fact that $ko_m(B\Z_p)$ is $q$-divisible and $ko_n(B\Z_q)$  is a $q$-torsion group.
\end{proof}
\begin{cor}\label{cor} 
For $m,n>2$ and odd relatively prime $p$ and $q$ or for $ p$ odd and $q$ even with $n=2k-1$ for even $k$ we have
$$\cat(L^m_p\times L^n_q)\le m+n-2.$$
\end{cor}
\begin{proof}
In this case the lens spaces are spin~\cite{Fr} and we can apply Proposition~\ref{vanish}. Then Theorem~\ref{cat=n-1} and the fact that $\pi_2(L_p^{m}\times L_{q}^{n})=0$
imply the result.
\end{proof}

For $m=n=3$ we obtain a different proof of Corollary~\ref{3}:
\begin{cor} 
$\cat(L^3_p\times L^3_q)=4$ for all relatively prime $p$ and $q$.
\end{cor}

\subsection{Coinvariants} The following lemma can be found in~[CE, Lemma 3.3].
\begin{lemma}\label{Lemma}
A commutative diagram with exact rows
$$
\begin{CD}
@. A' @>>> A @>>> A'' @>>> 0\\
@. @Vf'VV @VfVV @Vf''VV\\
0 @>>> C' @>>> C @>>> C'' @.\\
\end{CD}
$$
defines an exact sequence
$$
ker(f')\to ker(f)\to ker(f'')\to coker(f')\to coker(f)\to coker(f'').
$$
\end{lemma}

Let $p:E\pi\to B\pi$ be the universal covering. Thus $p$ is the projection onto the orbit space of
a free cellular $\pi$-action. Below we use the following abbreviations: $\pi=\pi_1(B)$, $B=B\pi$ and $E=E\pi$. 

\begin{prop}\label{n-1}
$p_*:\pi_n(E/E^{n-1})\to \pi_n(B/B^{n-1})$ is an epimorphism.
\end{prop}
\begin{proof}
Note that $\pi_n(B/B^{n-1})$ is generated by the spheres from the wedge of $n$-spheres $B^n/B^{n-1}=\vee S^n$.
For every such generator $\sigma^n/\partial\sigma^n$ for any lift $\bar\sigma^n$ of the $n$-simplex $\sigma^n$,
the quotient $\bar\sigma^n/\partial\bar\sigma^n$ defines a spheroid which is mapped to $\sigma^n/\partial\sigma^n$.
\end{proof}
Recall that $\pi_*^s$  denotes the stable homotopy groups.
\begin{cor}\label{epi} 
For $n\ge 5$ the induced homomorphism
$$p_*':\pi_n^s(E,E^{n-1})\to \pi_n^s(B,B^{n-1})$$ is an epimorphism.
\end{cor}
\begin{proof}
This follows from the obvious natural isomorphims
$$\pi_n(E/E^{n-1})=\pi_n^s(E/E^{n-1})=\pi_n(E,E^{n-1})$$ and $$\pi_n(B/B^{n-1})=\pi_n^s(B/B^{n-1})=\pi_n(B,B^{n-1}).$$ 
\end{proof}

\subsection{Proof of Theorem~\ref{coinvariants-mono}}
For $n\ge 5$,  the induced homomorphism 
$$p_*:\pi_n(B,B^{(n-2)})\to \pi_n(B/B^{(n-2)})$$ factors through the group of coinvariants, $p_*=\bar p_*\circ q_*$,
$$\pi_n(B,B^{(n-2)})\stackrel{q_*}\to\pi_n(B,B^{(n-2)})_{\pi}\stackrel{\bar p_*}\to\pi_n(B/B^{(n-2)})$$ where $\bar p_*$ is injective.

Note that for $n\ge 5$, $$\pi_n(B,B^{n-2})=\pi_n(E,E^{n-2})=\pi_n^s(E,E^{n-2}),\ \ \ \pi_n(B/B^{n-2})=\pi_n^s(B,B^{n-2}),$$
and the composition $$\pi_n(B,B^{(n-2)})\stackrel{q_*}\to\pi_n(B,B^{(n-2)})_{\pi}\stackrel{\bar p_*}\to\pi_n(B/B^{(n-2)})$$
coincides with 
$$\pi_n^s(E,E^{(n-2)})\stackrel{q_*}\to\pi_n^s(E,E^{(n-2)})_{\pi}\stackrel{\bar p_*}\to\pi_n^s(B,B^{(n-2)})$$
where $$\bar p_*\circ q_*=p_*:\pi_n^s(E,E^{(n-2)})\to \pi_n^s(B,B^{(n-2)})$$ is the homomorphism induced by the projection $p$.

Also note that $\pi_*^s(E,E^i)$ inherits a $\pi$-module structure via the $\pi$-action.

We extract from  the diagram generated by $p$ and exact  $\pi^s_*$-homology sequence of the triple $(E^n,E^{n-1},E^n)$ the following two diagrams:
$$
\begin{CD}
\pi_{n+1}^s(E^n,E^{n-1})  @>\bar j_{n+1}>> \pi^s_n(E^{n-1},E^{n-2}) @>>> \bar K   @>>> 0\\ 
@V{p_*^1}VV @V{p_*^2}VV @V{\alpha}VV \\
\pi^s_{n+1}(B^n,B^{n+1}) @>j_{n+1}>> \pi^s_n(B^{n-1},B^{n-2}) @>>> K @>>> 0\\
\end{CD}
$$
where $K,\bar K$ are cokernels of $j_{n+1},\bar j_{n+1}$ 
and
$$
\begin{CD}
0 @>>>\bar H @>>>\pi^s_n(E^n,E^{n-1})
@>\bar j_n>> \pi^s_{n-1}(E^{n-1},E^{n-2})\\
@. @V{\beta}VV @V{p_*^3}VV @V{p_*^4}VV \\
0 @>>>H @>>>\pi^s_n( B^n,B^{n-1})
@>j_n>> \pi^s_{n-1}(B^{n-1},B^{n-2})\\
\end{CD}
$$
where $H$ and $\bar H$ are kernels of $j_n$ and $\bar j_n$.
Note that the homomorphisms $p^3_*$ and $p^4_*$ are the direct sums of the augmentation homomorphism $$\epsilon:\Z\pi\to\Z.$$
The homomorphisms $p_*^1$ and $p_*^2$ are  direct sums of the mod 2 augmentation homomorphisms $$\bar\epsilon:\Z_2\pi\to\Z_2.$$
Also note that $p_*^i\otimes_{\pi}1_{\Z}$ is an isomorphism for $i=1,2,3,4$. 
Taking the tensor product of the first diagram with $\Z$ over $\Z\pi$ would give a commutative diagram with two left vertical arrows isomorphism. Then by the Five Lemma
$\alpha'=\alpha\otimes_{\pi}1_{\Z}$ is an isomorphism.

\

We argue that $\beta'=\beta\otimes_{\pi}1_{\Z}$ is a monomorphism. 
Note that $\ker(\beta)\subset \ker(p^3_*)=\oplus I(\pi)$ where
$I(\pi)$ is the augmentation ideal.

{\it Claim:} $\ker(\beta)\otimes_{\pi}\Z=0$. 

{\it Proof}: We show that $x\otimes_{\pi} 1=0$ for all $x\in L$. Let $x= \sum x_i$, $x_i\in I(\pi)$. It suffices to show that $x_i\otimes_{\pi}1=0$
for all $x_i$. Note that $x_i=\sum n_j(\gamma_j-e)$, $\gamma_j\in \pi$, $n_J\in\Z$. Note that $(\gamma-e)\otimes_{\pi} 1=0$ since
 $$(\gamma-e)\otimes_{\pi} 1=\gamma\otimes_{\pi} 1- e\otimes_{\pi} 1= \gamma(e\otimes_{\pi} 1)- e\otimes_{\pi} 1= e\otimes_{pi} \gamma(1)- e\otimes_{\pi} 1=0.$$

The tensor product with $\Z$ over $\Z\pi$ of the exact sequence
$$
\ker(\beta)\to \bar H\to \im(\beta)\to 0$$
implies that $$\beta_0=\beta\otimes id:\bar H\otimes_{\pi}\Z=\bar H_{\pi}\to \im(\beta)\otimes_{\pi}\Z=\im(\beta)$$ is an isomorphism.
Then $\beta'$ is a monomorphism as the composition of an isomorphism $\beta_0$ and the inclusion $\im(\beta)\to H$.

\

We consider  the diagram of short exact sequences:
$$
\begin{CD}
0 @>>> \bar K @>\bar\phi>> \pi^s_n(E^n,E^{n-2}) @>\bar\psi>> \bar H @>>> 0\\
@. @V{\alpha}VV @Vp_*VV @V{\beta}VV @.\\
0 @>>> K @>\phi>> \pi^s_n(B^n,B^{n-2}) @>\psi>> H @>>> 0\\
\end{CD}
$$

Then we apply the tensor product with $\Z$ over $\Z\pi$ to this diagram to obtain the following
commutative diagram with exact rows
$$
\begin{CD}
 @. \bar K_{\pi} @>\bar\phi>> \pi^s_n(E^n,E^{n-2})_{\pi} @>\bar\psi>> \bar H_{\pi} @>>> 0\\
@. @V{\alpha'}VV @V{\tilde p_*}VV @V{\beta'}VV @.\\
0 @>>> K @>\phi>> \pi^s_n(B^n,B^{n-2}) @>\psi>> H .\\
\end{CD}
$$

Lemma~\ref{Lemma} implies tat $\tilde p_*$ is a monomorphism.

\

Next we consider the diagram generated by  the triples $(E,E^n,E^{n-2})$ and
$(B,B^n,B^{n-2})$, 
$$
\begin{CD}
\pi^s_{n+1}(E,E^n) @>>>\pi_n^s(E^n, E^{n-2}) @>>> \pi_n^s(E,E^{n-2}) @>>> 0\\
@VVV @VVV @VVV @.\\
\pi^s_{n+1}(B,B^n) @>>>\pi_n^s(B^n, B^{n-2}) @>>> \pi_n^s(B,B^{n-2}) @>>> 0\\
\end{CD}
$$
and
tensor it with $\Z$ over $\Z\pi$ to obtain the following commutative diagram with exact rows
$$
\begin{CD}
\pi^s_{n+1}(E,E^n)_{\pi} @>>>\pi_n^s(E^n, E^{n-2})_{\pi} @>>> \pi_n^s(E,E^{n-2})_{\pi} @>>> 0\\
@Vp'_*VV @V\tilde p_*VV @V\bar p_*VV @.\\
\pi^s_{n+1}(B,B^n) @>>>\pi_n^s(B^n, B^{n-2}) @>>> \pi_n^s(B,B^{n-2}) @>>> 0\\
\end{CD}
$$
Since $p'_*$ is an epimorphism (see Corollary~\ref{epi}) and $\tilde p_*$  is a monomorphism by  the monomorphism version of
the Five Lemma we obtain that $\bar p_*$ is a monomorphism. \qed

\section{On the category of the sum}

The following theorem was proven by R. Newton 
under the assumption that $\cat M,\cat N>2$~\cite{N}. 
\begin{thm}\label{th1}
For closed manifolds $M$ and $N$ there is the inequality,
$$
\cat(M\# N)\le\max\{\cat M,\cat N\}.
$$
\end{thm}
His proof is based on  obstruction theory.
Here we present a proof that works in full generality.
Our proof is an application of the following

\begin{thm}[W. Singhof~\cite{S} Theorem 4.4]\label{singhof}
For any closed $n$-manifold $M$ with $\cat M=k\ge 2$, there is a categorical partition $Q_0,\dots,Q_k$ into
manifolds with boundary such that $Q_i\cap Q_j$ is a $(n-1)$-manifold with boundary (possibly empty) for all $i,j$ and each $Q_i$
admits a deformation retraction onto a $(n-k)$-dimensional CW complex.
\end{thm}

For $B\subset A\subset X$, a homotopy $H:A\times I\to X$ is called a {\em deformation} of $A$ in $X$ onto $B$ if $H_{A\times \{0\}}=1_A$,
$H(A\times \{1\})=B$, and $H(b,t)=b$ for all $b\in B$ and $t\in I=[0,1]$.
The following is well-known.
\begin{prop}\label{2.2}
Let $A\subset M$ be a subset contractible to a point in an $m$-manifold $M$ and let $B\subset A$ be a  closed $n$-ball which admits a regular neighborhood.
Then there is a deformation of $A$ in $M$ onto $B$.
\end{prop}

{\bf Proof of Theorem~\ref{th1}.} Let $n=\dim M=\dim N$. Suppose that $\cat M,\cat N\le k$. We show that $\cat(M\#N)\le k$.
If $k=1$, the statement obviously follows from the fact that $M$ and $N$ are homeomorphic to the sphere. We assume that $k\ge 2$.
Let $Q_0,\dots Q_k$ be a partition of $M$ into $M$-contractible subsets as in Singhof's theorem. We may assume that $W_0\cap W_1\ne\emptyset$. Moreover, we may assume that there is a closed topological $n$-ball $D\subset M$ with a collar in $M$ and $D_0=D\cap W_0$, $D_1=D\cap W_1$ such that the triad $(D,D_0,D_1)$ is homeomorphic to the triad
$(B, B_+,B_-)$, where $B$ is the unit ball in $\mathbb R^n$, $B_+=B\cap{\mathbb R}_+^n$,  $B_-=B\cap\mathbb R_-^n$, $\mathbb R_+^n=\{(x_1,\dots x_n)\in\mathbb R^n\mid x_n\ge 0\}$ and
$\mathbb R_-^n=\{(x_1,\dots x_n)\in\mathbb R^n\mid x_n\le 0\}$ are the half-spaces. Similarly, we may assume that there is a categorical partition
$V_0,\dots, V_k$ of $N$ as in Theorem~\ref{singhof} and a closed $n$-ball $D'$ with a collar such that the triad $(D',D'_0,D'_1)$ is homeomorphic to the triad
$(B, B_+,B_-)$, where $D'_0=D'\cap V_0$, $D'_1=D'\cap V_1$. 

We may assume that the connected sum $M\# N$
is realized as a subset $M\# N=M\cup N\setminus \Int D\subset M\cup_hN$  for some homeomorphism
$h:D'\to D$ that preserves the triad structures.

Let $W_0=(Q_0\setminus \Int D)\cup (V_0\setminus \Int D')$, $W_1=(Q_1\setminus \Int D)\cup (V_1\setminus \Int D')$, and
$W_i=Q_i\cup V_i$ for $i=2,\dots, k$. Note that $Q_i\cap V_i=\emptyset$ for $i\ge 2$. By Singhof's theorem each $Q_i$ can be deformed to an $(n-k)$-dimensional subset $S_i$ contractible in $M$. Since $k\ge 2$,  there is a contraction of $S_i$ to a point in $M$ that misses a given point. Hence, there is a contraction of $S_i$ to a point in $M$ that  misses the ball $D$.
Thus $Q_i$ for $i\ge 2$ can be contracted to a point in $M\# N$.
Similarly, for $i\ge 2$ the set $V_i$ can be contracted to a point in $M\# N$. Hence the sets $W_i$ for $i\ge 2$ are categorical.
Let $A_i= Q_i\cap \partial D$.
For $i=0,1$ we show that there is a deformation of $Q_i\setminus \Int D$ in $M\#N$ to $A_i$.
Since $Q_i\setminus \Int D$ is homeomorphic to $Q_i$, it can be deformed onto $(n-k)$-dimensional subset $S_i$ contractible in $M$.
A contraction of $S_i$ to a point can be chosen missing $c_0\in \Int D$. By Proposition~\ref{2.2} there is a deformation of $Q_i\setminus \Int D$
in $M\setminus\{c_0\}$ onto $A_i$ fixing $A_i$. Similarly, for $i=0,1$ there is a deformation of 
$V_i\setminus \Int D'$ 
in $N\setminus\{c_0'\}$ to $A_i=V_i\cap\partial D'$ fixing $A_i$ where $c_0'\in \Int D'$. 
Applying the radial projections from $c_0$ and $c_0'$ gives us such deformations in $M\# N$. Pasting these two deformations defines a deformation of $W_i$
$i=0,1$  in $M\#N$ to $A_i$. Since the sets $A_i$ are contractible, it follows that the sets $W_i$, $i=0,1$ are categorical.\qed

\end{document}